\documentclass[aap]{imsart}

\RequirePackage{amsthm,amsmath,amsfonts,amssymb}
\RequirePackage[numbers]{natbib}
\usepackage{amsthm,amssymb,amsmath,mathtools,enumerate,tabularx}
\usepackage{bm,bbm}
\RequirePackage[colorlinks,citecolor=blue,urlcolor=blue]{hyperref}
\usepackage{cleveref}

\startlocaldefs
\newtheorem*{theorem*}{Theorem}
\newtheorem{theorem}{Theorem}
\crefname{theorem}{Theorem}{Theorems}
\newtheorem{proposition}{Proposition}
\crefname{proposition}{Proposition}{Propositions}
\newtheorem{lemma}{Lemma}
\crefname{lemma}{Lemma}{Lemmas}
\newtheorem{corollary}{Corollary}
\crefname{corollary}{Corollary}{Corollaries}

\crefname{section}{Section}{Sections}

\theoremstyle{remark}

\theoremstyle{definition}
\newtheorem{definition}{Definition}
\crefname{definition}{Definition}{Definitions}

\DeclareMathOperator{\re}{Re}

\newcolumntype{M}[1]{l>{\hbox to #1\bgroup\hss$}l<{$\egroup}}

\makeatletter
\newcommand\@brcolwidth{1em}
\newenvironment{blmatrix}{%
    \left[%
    \hskip-\arraycolsep
    \new@ifnextchar[\@brarray{\@brarray[\@brcolwidth]}%
}{%
    \endarray
    \hskip -\arraycolsep
    \right]%
}
\def\@brarray[#1]{\array{l*\c@MaxMatrixCols {M{#1}}}}
\makeatother

\endlocaldefs

\begin{document}

\begin{frontmatter}
\title{The Multivariate Rate of Convergence for\\ Selberg's Central Limit Theorem}
\runtitle{The Rate of Convergence for Selberg's CLT}

\begin{aug}
\author[A]{\fnms{Asher}~\snm{Roberts}\ead[label=e1]{aroberts@gradcenter.cuny.edu}}
\address[A]{Graduate Center, City University of New York, New York, NY 10016\printead[presep={,\ }]{e1}}
\end{aug}

\begin{abstract}
In this paper we quantify the rate of convergence in Selberg's central limit theorem for \(\log|\zeta(1/2+it)|\) based on the method of proof given by Radziwi\l\l\ and Soundararajan in \cite{rands}. We achieve the same rate of convergence of \((\log\log\log T)^2/\sqrt{\log\log T}\) as Selberg in \cite{selbergamalfi} in the Kolmogorov distance by using the Dudley distance instead. We also prove the theorem for the multivariate case given by Bourgade in \cite{mesoscopic} with the same rate of convergence as in the single variable case.
\end{abstract}



\end{frontmatter}

\section{Introduction}

Let \(\tau\) be a random point distributed uniformly on \([T,2T]\). Selberg's central limit theorem states that as \(T\to\infty\), 

\[\mathbb{P}\left(\log|\zeta(\tfrac{1}{2}+i\tau)|> V\sqrt{\tfrac{1}{2}\log\log T}\right)\to\frac{1}{\sqrt{2\pi}}\int_V^\infty e^{-u^2/2}du,\quad \forall V\in\mathbb{R}.\]

\noindent We adapt the techniques of Radziwi\l\l\ and Soundararajan in \cite{rands} that proved Selberg's central limit theorem for \(\log|\zeta(1/2+it)|\) to prove \cref{thm:selbergcltrate}, and consequently prove the multivariate theorem of Bourgade \cite{mesoscopic} stated in \cref{cor:multivariateselbergclt}.
	 
\subsection{Main Results}

\begin{definition}
The Dudley (bounded Wasserstein) distance between two random vectors in \(\mathbb{R}^n\) defined on a probability space \((\Omega,\mathcal{F},\mathbb{P})\) is given by \[d_{\mathcal{D}}(X,Y)=\sup_{f\in\mathcal{L}}|\mathbb{E}[f(X)]-\mathbb{E}[f(Y)]|\] where \(\mathcal{L}=\{f:\mathbb{R}^n\to\mathbb{R}\mid f\) is Lipschitz, \(\|f\|_\infty\leq 1\) and \(\|f\|_{\text{Lip}}\leq 1\}\). Here and throughout the paper \(\|f\|_\infty\) denotes the supremum norm of \(f\) and \(\|f\|_{\text{Lip}}=\sup_{x\ne y}\frac{|f(x)-f(y)|}{\|x-y\|_2}\) with \(\|\cdot\|_2\) denoting the \(L^2\) norm.
\end{definition}	

We use the Dudley distance (e.g. see \cite{originaldudley} and \cite{dudley}) to estimate the rate of convergence. The boundedness property of this metric helps us when estimating the error associated with mollifying \(\zeta(s)\) and connecting our mollifer to prime sums (\cref{prop:mollifying,prop:approximatingthemollifier}). The Lipschitz property helps us when computing the remaining estimates needed, since computing the error associated with moving off axis and with truncating prime sums (\cref{prop:movingoffaxis,prop:truncatingprimesum}) would be more difficult using the Kolmogorov distance.\\

The main result of the paper is \cref{thm:selbergcltrate}. We use the Vinogradov notation \(f\ll g\) to mean \(f= O(g)\) as \(T\to\infty\).
	 
\begin{theorem}\label{thm:selbergcltrate}
Let \(\tau\) be a random point distributed uniformly on \([T,2T]\), and let \(|h-h'|\sim(\log T)^{-\alpha}\) for \(\alpha\in(0,1)\). Then for \(T\) large enough, \[d_{\mathcal{D}}\left[\left(\frac{\log|\zeta(\tfrac{1}{2}+i(\tau+h))|}{\sqrt{\frac{1}{2}\log\log T}},\frac{\log|\zeta(\tfrac{1}{2}+i(\tau+h'))|}{\sqrt{\frac{1}{2}\log\log T}}\right),\left(\mathcal{Z},\mathcal{Z}'\right)\right]\ll\frac{(\log\log\log T)^2}{\sqrt{\log\log T}},\]
where \((\mathcal{Z},\mathcal{Z}')\) is a Gaussian vector with mean \(0\) and covariance matrix \(\begin{pmatrix}
1 & \alpha \\
\alpha & 1
\end{pmatrix}\).
\end{theorem}

In \cite{selbergamalfi} Selberg claims a rate of convergence of \((\log\log\log T)^2/\sqrt{\log\log T}\) using the Kolmogorov distance, a proof of which is given by Tsang in \cite{tsang}. We have been able to replicate this rate using the Dudley distance, which suffices to prove Bourgade's theorem \cite{mesoscopic}.

\begin{corollary}\label{cor:multivariateselbergclt}
Let \(\tau\) be a random point distributed uniformly on \([T,2T]\). Then for any measurable subset \(\mathcal{B}\subset\mathbb{R}^2\), as \(T\to\infty\), \[\mathbb{P}\left(\Big(\log|\zeta(\tfrac{1}{2}+i(\tau+h))|,\log|\zeta(\tfrac{1}{2}+i(\tau+h'))|\Big)\in\sqrt{\tfrac{1}{2}\log\log T}\,\mathcal{B}\right)\to\int_{\mathcal{B}}\frac{e^{-(x^TC^{-1}x)/2}}{2\pi\sqrt{\det C}}\,dx,\]
where \(|h-h'|\sim(\log T)^{-\alpha}\) for \(\alpha\in(0,1)\) and \(C\) is the covariance matrix \(\begin{pmatrix}
1 & \alpha \\
\alpha & 1
\end{pmatrix}\).
\end{corollary}

\subsection{Relation to Previous Results}
The multivariate extension of Selberg's central limit theorem for \(\log\zeta(\frac{1}{2}+it)\) was proven by Bourgade in \cite{mesoscopic} using techniques similar to the ones Selberg originally used to prove his theorem in \cite{selberg}. Later, Radziwi\l\l\ and Soundararajan used a more streamlined technique to prove Selberg's theorem for \(\log|\zeta(\frac{1}{2}+it)|\) in \cite{rands}, which we adapt to prove \cref{cor:multivariateselbergclt}. This presents a new proof for the multivariate extension given in \cite{mesoscopic} in the real case. Our work is ill-suited for proving the theorem in the imaginary case for the same reason as in \cite{rands}. That is, the difficulty lies in mollifying the imaginary part of \(\zeta\) with the method used in \cref{prop:mollifying}. Our work can also be extended to obtain analogous results for Dirichlet \(L\)-functions, as Hsu and Wong did in \cite{wongandhsu} to prove the central limit theorem in the single variable case using the techniques of Radziwi\l\l\ and Soundararajan.

Some modifications to the approach of \cite{rands} are needed to prove \cref{thm:selbergcltrate}. In \cite{rands} the parameter \(W\) used to move off axis is taken to be \((\log\log\log T)^4\). However, this results in the worse rate of convergence of \[\frac{(\log\log\log T)^4}{\sqrt{\log\log T}}\] when applying our techniques, prompting us to move closer to the critical line and take \(W=(\log\log\log T)^2\). We then have to correspondingly modify the parameters \(X\) and \(Y\) in our mollifer \(M\) to account for this change when mollifying \(\zeta\) in \cref{lma:zetaMestimate}. Achieving the rate of convergence given in \cref{thm:selbergcltrate} necessitates a few other changes as well. When approximating \(\log|M^{-1}(s)|\) with \(\re\mathcal{P}(s)\) in \cref{prop:approximatingthemollifier} we look directly at the difference \(\exp(-\mathcal{P}(s))-M(s)\), bypassing the Dirichlet series approximation for \(\exp(-\mathcal{P}(s))\) given by \(\mathcal{M}(s)\) in Proposition 3 of \cite{rands}. This allows us to more precisely estimate the associated error. We also employ a technique of Arguin, Belius, Bourgade, Radziwi\l\l , and Soundararajan in \cite{maxzeta} to compute \(\lfloor\log\log T\rfloor\) moments for \(\mathcal{P}(s)\) in \cref{lma:largevaluesofP}, since this gives us the moments of a linear combination of \(\mathcal{P}(s)\). This allows us to calculate the rate of convergence in a multivariate context. The moments of a linear combination also provide a more direct proof of the real case in \cite{mesoscopic} using the Cram\'er-Wald device.

Improving upon the rate of convergence in \cite{selbergamalfi} using our techniques is difficult, since it requires taking \(W=o((\log\log\log T)^2)\), which is not large enough to handle our parameter \(X\) when mollifying \(\zeta\) in \cref{lma:zetaMestimate} (adjusting \(X\) to handle \(W\) is essentially equivalently difficult, as this results in the mollifer becoming too large). Selberg also obtains a rate of \(\log\log\log T/\sqrt{\log\log T}\) for the imaginary part of \(\log\zeta\), and an explicit proof of this is given in \cite{wahl} in the context of mod-gaussian convergence (more information is given in Kowalski and Nikeghbali \cite{modgaussian}). In addition, there is numerical data by Keating and Snaith \cite{keatingsnaith} that corroborates that the current rate of convergence for the imaginary part of zeta is better than that of the real part. Our techniques in and of themselves are ill-suited to obtain this result, again due to the difficulty in mollifying the imaginary part of \(\zeta\) in \cref{prop:mollifying}.

In random matrix theory, Bourgade, Hughes, Nikeghbali, and Yor show in \cite{randommatrix} that \[\left|\mathbb{P}\left(\frac{\log|\det(I-A)|}{\sqrt{\frac{1}{2}\log N}}\leq x\right)-\frac{1}{\sqrt{2\pi}}\int_{-\infty}^x e^{-t^2/2}\,dt\right|\leq\frac{C}{(\log N)^{3/2}(1+|x|)^3}\] where \(\det(I-A)\) is the characteristic polynomial of an \(N\times N\) matrix sampled under the Haar measure on the unitary group \(U(N)\), and \(C\) is a constant. This result might suggest a rate of convergence for Selberg's central limit theorem of \(1/(\log\log T)^{3/2}\). However, it seems from \cref{prop:comparingmoments} that it is unlikely for one to be able to obtain a better rate than \(1/\sqrt{\log\log T}\).

\subsection{Structure of the Proofs} All proofs are given in \cref{proofs}. The proof of \cref{thm:selbergcltrate} follows by seven successive propositions. The distances between the random variables are explicit at each step of the proof. The trickiest propositions to handle are \cref{prop:mollifying,prop:comparingmoments}, mollifying \(\zeta\) and comparing moments via Fourier transforms. However, the steps that incur the greatest errors are \cref{prop:movingoffaxis,prop:truncatingprimesum}, moving off axis and truncating the prime sum. Although these steps are proven in a straightforward fashion, it is difficult to improve them due to the impact of any modifications on \cref{prop:mollifying,prop:comparingmoments}. That is, using a longer mollifier would allow us to get closer to the critical line in \cref{prop:movingoffaxis}, but would consequently incur too great of an error when mollifying \(\zeta\) in \cref{prop:mollifying}. In addition, using a longer prime sum would improve the error in \cref{prop:truncatingprimesum}, but then the error in comparing Fourier transforms in \cref{prop:comparingmoments} would be too large.\\

Let \(\tau\) be a random point distributed uniformly on \([T,2T]\) and set \begin{equation}
\sigma_0=\frac{1}{2}+\frac{W}{\log T},\quad s_0=\sigma_0+i(\tau+h),\text{ and }s_0'=\sigma_0+i(\tau+h').\label{eq:sigma0}
\end{equation} It is necessary to take \(W=o(\sqrt{\log\log T})\), in order to typically approximate \(\log|\zeta(\frac{1}{2}+i(\tau+h))|\) and \(\log|\zeta(\frac{1}{2}+i(\tau+h'))|\) by \(\log|\zeta(s_0)|\) and \(\log|\zeta(s_0')|\), respectively, as in Proposition 1 of \cite{rands}. There is some flexibility in the choice of \(W\), but as we will see in \cref{prop:movingoffaxis} below, \(W\) is the primary contributor to the overall rate of convergence that we obtain, so we take \begin{equation}
W=K(\log\log\log T)^2,\label{eq:W}
\end{equation} where \(K\) is a large constant. To simplify notation, let \[\mathfrak{s}=\sqrt{\tfrac{1}{2}\log\log T}\] denote the standard deviation appearing in \cref{thm:selbergcltrate}. The contribution towards the rate of convergence in \cref{thm:selbergcltrate} by moving off axis is then given by
\begin{proposition}[Moving off axis]\label{prop:movingoffaxis}
With the notation above, we have
\[d_{\mathcal{D}}(\mathbf{V},\mathbf{W})\ll\frac{(\log\log\log T)^2}{\sqrt{\log\log T}},\] where \(\mathbf{V}=(\mathcal{V},\mathcal{V}')\), \(\mathbf{W}=(\mathcal{W},\mathcal{W}')\) are intermediary variables defined by \begin{align*}
\mathcal{V}&=\frac{\log|\zeta(\tfrac{1}{2}+i(\tau+h))|}{\mathfrak{s}} & \mathcal{W}&=\frac{\log|\zeta(s_0)|}{\mathfrak{s}}\\
\mathcal{V}'&=\frac{\log|\zeta(\tfrac{1}{2}+i(\tau+h'))|}{\mathfrak{s}} & \mathcal{W}'&=\frac{\log|\zeta(s_0')|}{\mathfrak{s}}.
\end{align*}
\end{proposition}
We now take \[X=T^{1/(K'\log\log\log T)}\text{ and }Y=T^{1/(K'\log\log T)}\] where \(K'\) is a large constant such that \(2<K'<K\) (for the \(K\) in equation \eqref{eq:W}), and a mollifier \[M(s)=\sum_n\frac{\mu(n)a(n)}{n^s}\] with \[a(n)=\begin{cases}1 & \text{if }p\mid n\Rightarrow p\leq X,\\
& \text{and }\Omega(n)\leq 100\log\log T\text{ for }p\leq Y,\\
& \text{and }\Omega(n)\leq 100\log\log\log T\text{ for }Y\leq p\leq X,\\
0 & \text{otherwise},\end{cases}\] where \(p\) is prime and \(\Omega(n)\) is the number of prime factors of \(n\) (with multiplicity). The constant \(K'\) is needed in the definitions of the parameters \(X\) and \(Y\) to control the size of the mollifier, and this constant is what necessitates the insertion of a larger constant in our choice of \(W\). The contribution towards the rate of convergence in \cref{thm:selbergcltrate} by mollifying is then given by
\begin{proposition}[Mollifying]\label{prop:mollifying}
With the notation above, we have \[d_{\mathcal{D}}(\mathbf{W},\mathbf{X})\ll\frac{1}{\sqrt{\log\log T}},\]
where \(\mathbf{W}=(\mathcal{W},\mathcal{W}')\) is defined as in \cref{prop:movingoffaxis} and \(\mathbf{X}=(\mathcal{X},\mathcal{X}')\) is the intermediary variable given by \[\mathcal{X}=\frac{\log|M^{-1}(s_0)|}{\mathfrak{s}} \qquad\text{and}\qquad \mathcal{X}'=\frac{\log|M^{-1}(s_0')|}{\mathfrak{s}}.\]
\end{proposition}

We approximate the mollifier \(M(s)\) with the sum \[\mathcal{P}(s)=\sum_{2\leq n\leq X}\frac{\Lambda(n)}{n^s\log n},\] where \(\Lambda(n)\) is the von Mangoldt function given by \[\Lambda(n)=\begin{cases}\log p & \text{if }n=p^k,\\ 0 & \text{otherwise}.\end{cases}\] Then the contribution towards the rate of convergence in \cref{thm:selbergcltrate} by using this approximation is
\begin{proposition}[Approximating the mollifier]\label{prop:approximatingthemollifier}
With the notation above, we have
\[d_{\mathcal{D}}(\mathbf{X},\mathbf{Y})\ll(\log\log T)^{-80},\]
where \(\mathbf{X}=(\mathcal{X},\mathcal{X}')\) is defined as in \cref{prop:mollifying} and \(\mathbf{Y}=(\mathcal{Y},\mathcal{Y}')\) is the intermediary variable given by \[\mathcal{Y}=\frac{\re \mathcal{P}(s_0)}{\mathfrak{s}}\qquad\text{and}\qquad\mathcal{Y}'=\frac{\re \mathcal{P}(s_0')}{\mathfrak{s}}.\]
\end{proposition}

We then discard the higher order primes in the sum \(\re\mathcal{P}(s)\), electing to instead compare the moments of \[P(s)=\re\sum_{p\leq X}\frac{1}{p^s}\] with those of a Gaussian random variable. The higher order primes in \(\mathcal{P}(s)\) have a contribution towards the rate of convergence in \cref{thm:selbergcltrate} given by
\begin{proposition}[Discarding higher order primes]\label{prop:discardingprimes}
With the notation above, we have
\[d_{\mathcal{D}}\left[\mathbf{Y},\left(\frac{P(s_0)}{\mathfrak{s}},\frac{P(s_0')}{\mathfrak{s}}\right)\right]\ll\frac{1}{\sqrt{\log\log T}},\] where \(\mathbf{Y}=(\mathcal{Y},\mathcal{Y}')\) is the intermediary variable defined in \cref{prop:approximatingthemollifier}.
\end{proposition}

Next we split \(P(s)\), setting \[P_1(s)=\re\sum_{p\leq Y}\frac{1}{p^s} \qquad\text{and}\qquad P_2(s)=\re\sum_{Y\leq p\leq X}\frac{1}{p^s},\] so that \(P(s)=P_1(s)+P_2(s)\). In comparing the Fourier transform of \(P(s)\) to that of a Gaussian vector in \cref{prop:comparingmoments}, we incur an error of order \(1/\log\log T\) for \(P_1(s)\), which is absorbed by our normalization of \(1/\sqrt{\log\log T}\), and an error of \(1/\log\log\log T\) for \(P_2(s)\), which is greater than our normalization. As a result, we use \cref{prop:truncatingprimesum} to work with just \(P_1(s)\) when comparing moments, because the additional error generated by dropping \(P_2(s)\) is smaller than the error from moving off-axis in \cref{prop:movingoffaxis}. At the same time, we adjust our normalization \(\mathfrak{s}\) slightly, in order to make it align with the variance of \(P_1(s)\). The adjusted normalization \(\tilde{\mathfrak{s}}\) helps us match moments more exactly when calculating the difference in truncated moments in \cref{lma:fourierexpectations}, whereas dividing by \(\mathfrak{s}\) when computing the difference in moments would result in an unmanageable error term, even in the single variable case.

\begin{proposition}[Truncating the prime sum]\label{prop:truncatingprimesum}
With the notation above, we have
\[d_{\mathcal{D}}\left[\left(\frac{P(s_0)}{\mathfrak{s}},\frac{P(s_0')}{\mathfrak{s}}\right),\left(\frac{P_1(s_0)}{\tilde{\mathfrak{s}}},\frac{P_1(s_0')}{\tilde{\mathfrak{s}}}\right)\right]\ll\frac{\sqrt{\log\log\log T}}{\sqrt{\log\log T}},\]
where \(\tilde{\mathfrak{s}}^2=\frac{1}{2}\sum_{p\leq Y}1/p\).
\end{proposition}

The contribution towards the rate of convergence in \cref{thm:selbergcltrate} by approximating a normal distribution with \(P_1(s)\) is
\begin{proposition}[Comparing moments]\label{prop:comparingmoments}
With the notation above, we have
\[d_{\mathcal{D}}\left[\left(\frac{P_1(s_0)}{\tilde{\mathfrak{s}}},\frac{P_1(s_0')}{\tilde{\mathfrak{s}}}\right),\mathbf{\tilde{Z}}\right]\ll\frac{1}{\sqrt{\log\log T}},\] where \(\mathbf{\tilde{Z}}=(\mathcal{\tilde{Z}},\mathcal{\tilde{Z}}')\) is a Gaussian vector with mean \(0\) and the same covariance matrix as \((P_1(s_0)/\tilde{\mathfrak{s}},P_1(s_0')/\tilde{\mathfrak{s}})\).
\end{proposition}

Finally, we move from the normal approximation in \cref{prop:comparingmoments} to the normal distribution stated in \cref{thm:selbergcltrate} by showing that the contribution towards the rate of convergence in \cref{thm:selbergcltrate} by approximating the normal distribution \(\mathbf{Z}\) with \(\mathbf{\tilde{Z}}\) is
\begin{proposition}[Comparing normals]\label{prop:comparingnormals}
With the notation above, we have
\[d_{\mathcal{D}}\left[\mathbf{\tilde{Z}},\mathbf{Z}\right]\ll\frac{1}{\log\log T},\] where \(\mathbf{Z}=(\mathcal{Z},\mathcal{Z}')\) is a Gaussian vector with mean \(0\) and covariance matrix \(C=\begin{pmatrix}
1 & \alpha \\
\alpha & 1
\end{pmatrix}\).
\end{proposition}

\section{Proofs}\label{proofs}

\subsection{Proof of \cref{thm:selbergcltrate}}\begin{proof}
With the notation above, we have, by \crefrange{prop:movingoffaxis}{prop:comparingnormals} and the triangle inequality, \begin{align*}
&d_{\mathcal{D}}\left[\left(\frac{\log|\zeta(\tfrac{1}{2}+i(\tau+h))|}{\sqrt{\frac{1}{2}\log\log T}},\frac{\log|\zeta(\tfrac{1}{2}+i(\tau+h'))|}{\sqrt{\frac{1}{2}\log\log T}}\right),\left(\mathcal{Z},\mathcal{Z}'\right)\right]\\
&\ll d_{\mathcal{D}}(\mathbf{V},\mathbf{W})+d_{\mathcal{D}}(\mathbf{W},\mathbf{X})+d_{\mathcal{D}}(\mathbf{X},\mathbf{Y})+d_{\mathcal{D}}\left[\mathbf{Y},\left(\frac{P(s_0)}{\mathfrak{s}},\frac{P(s_0')}{\mathfrak{s}}\right)\right]\\
&\phantom{\ll}+d_{\mathcal{D}}\left[\left(\frac{P(s_0)}{\mathfrak{s}},\frac{P(s_0')}{\mathfrak{s}}\right),\left(\frac{P_1(s_0)}{\tilde{\mathfrak{s}}},\frac{P_1(s_0')}{\tilde{\mathfrak{s}}}\right)\right]\\
&\phantom{\ll}+d_{\mathcal{D}}\left[\left(\frac{P_1(s_0)}{\tilde{\mathfrak{s}}},\frac{P_1(s_0')}{\tilde{\mathfrak{s}}}\right),\mathbf{\tilde{Z}}\right]+d_{\mathcal{D}}\left[\mathbf{\tilde{Z}},\mathbf{Z}\right]\\
&\ll\frac{(\log\log\log T)^2}{\sqrt{\log\log T}}.
\end{align*}
\end{proof}

\subsection{Proof of \cref{cor:multivariateselbergclt}}\begin{proof}
Once \cref{thm:selbergcltrate} is proved, \cref{cor:multivariateselbergclt} follows immediately by taking \(T\to\infty\) in \cref{thm:selbergcltrate}, since convergence in the Dudley distance implies convergence in distribution. Full details are provided by Dudley in \cite{originaldudley}, where it is proven that the Dudley distance metrizes the weak\(^\ast\) topology.
\end{proof}

\subsection{Proof of \cref{prop:movingoffaxis}} \begin{proof}
We have
\[d_{\mathcal{D}}(\mathbf{V},\mathbf{W})=\sup_{f\in\mathcal{L}}|\mathbb{E}[f(\mathbf{V})]-\mathbb{E}[f(\mathbf{W})]|\leq\sup_{f\in\mathcal{L}}\mathbb{E}[|f(\mathbf{V})-f(\mathbf{W})|].\] Since \(f:\mathbb{R}^2\to\mathbb{R}\) where \(f\in\mathcal{L}\) is Lipschitz with \(\|f\|_{\text{Lip}}\leq 1\), we have \(|f(v)-f(w)|\leq\|f\|_{\text{Lip}}\|v-w\|_2\leq\|v-w\|_1\) with \(\|\cdot\|_1\) denoting the \(L^1\) norm, and thus \begin{equation}
|f(\mathbf{V})-f(\mathbf{W})|\leq\|\mathbf{V}-\mathbf{W}\|_1.\label{eq:lipschitzbound}
\end{equation} Therefore, by definition of the Dudley distance, the above distance is less than or equal to \begin{align*}
\mathbb{E}[\|\mathbf{V}-\mathbf{W}\|_1]&=\mathbb{E}[|\mathcal{V}-\mathcal{W}|+|\mathcal{V}'-\mathcal{W}'|]\\
&=\mathbb{E}[|\mathcal{V}-\mathcal{W}|]+\mathbb{E}[|\mathcal{V}'-\mathcal{W}'|],
\end{align*} by the linearity of expectation. We now introduce Proposition 1 of \cite{rands} as \cref{lma:offaxis}.

\begin{lemma}\label{lma:offaxis}
With the notation above, for \(T\) large enough and any \(\sigma>1/2\), we have \[\mathbb{E}\Big[\Big|\log|\zeta(\tfrac{1}{2}+i\tau)|-\log|\zeta(\sigma+i\tau)|\Big|\Big]\ll(\sigma-\tfrac{1}{2})\log T.\]
\end{lemma}

\noindent The proof of \cref{lma:offaxis} uses Hadamard's factorization formula for the completed \(\zeta\)-function and is given in \cite{rands}. The zeros of \(\zeta(s)\) are mentioned only when using the Riemann-von Mangoldt formula to evaluate the resulting integral. By \cref{lma:offaxis}, both \(\mathbb{E}[|\mathcal{V}-\mathcal{W}|]\) and \(\mathbb{E}[|\mathcal{V}'-\mathcal{W}'|]\) are \(\ll W/\sqrt{\log\log T}\), and so our result follows.
\end{proof}

\subsection{Proof of \cref{prop:mollifying}} \begin{proof}
Similar to the way we proved \cref{prop:movingoffaxis}, we have
\[d_{\mathcal{D}}(\mathbf{W},\mathbf{X})=\sup_{f\in\mathcal{L}}|\mathbb{E}[f(\mathbf{W})]-\mathbb{E}[f(\mathbf{X})]|\leq\sup_{f\in\mathcal{L}}\mathbb{E}[|f(\mathbf{W})-f(\mathbf{X})|].\] 
We then evaluate the expectation by splitting it into separate cases, so that the last term is
\begin{align}
\ll&\sup_{f\in\mathcal{L}}\left\{\mathbb{E}\left[|f(\mathbf{W})-f(\mathbf{X})|\cdot\mathbbm{1}\!\left(|\zeta(s_0)M(s_0)-1|\leq\frac{1}{2},|\zeta(s_0')M(s_0')-1|\leq\frac{1}{2}\right)\right]\right.\label{eq:wxleq}\\
&+\mathbb{E}\left[|f(\mathbf{W})-f(\mathbf{X})|\cdot\mathbbm{1}\!\left(|\zeta(s_0')M(s_0')-1|>\frac{1}{2}\right)\right]\label{eq:wxprimegreater}\\
&+\left.\mathbb{E}\left[|f(\mathbf{W})-f(\mathbf{X})|\cdot\mathbbm{1}\!\left(|\zeta(s_0)M(s_0)-1|>\frac{1}{2}\right)\right]\right\}.\label{eq:wxgreater}
\end{align}
As we obtained the inequality \eqref{eq:lipschitzbound} by taking advantage of the Lipschitz nature of \(f\), we similarly obtain \[|f(\mathbf{W})-f(\mathbf{X})|\leq\|\mathbf{W}-\mathbf{X}\|_1.\] Therefore, the expectation in \eqref{eq:wxleq} is less than or equal to \begin{align*}
&\mathbb{E}\left[\|\mathbf{W}-\mathbf{X}\|_1\cdot\mathbbm{1}\!\left(|\zeta(s_0)M(s_0)-1|\leq\frac{1}{2},|\zeta(s_0')M(s_0')-1|\leq\frac{1}{2}\right)\right]\\
&\leq\mathbb{E}\left[|\mathcal{W}-\mathcal{X}|\cdot\mathbbm{1}\!\left(|\zeta(s_0)M(s_0)-1|\leq\frac{1}{2}\right)+|\mathcal{W}'-\mathcal{X}'|\cdot\mathbbm{1}\!\left(|\zeta(s_0')M(s_0')-1|\leq\frac{1}{2}\right)\right]\\
&=\frac{1}{\mathfrak{s}}\mathbb{E}\left[|\log|\zeta(s_0)M(s_0)||\cdot\mathbbm{1}\!\left(|\zeta(s_0)M(s_0)-1|\leq\frac{1}{2}\right)\right]\\
&\phantom{\ll}+\frac{1}{\mathfrak{s}}\mathbb{E}\left[|\log|\zeta(s_0')M(s_0')||\cdot\mathbbm{1}\!\left(|\zeta(s_0')M(s_0')-1|\leq\frac{1}{2}\right)\right].
\end{align*}
On the event \(|\zeta(s)M(s)-1|\leq 1/2\) we have \(|\log|\zeta(s)M(s)||\leq\log 2\), and so both summands are \(\ll 1/\sqrt{\log\log T}\).
Since \(f\) is bounded, the expectations in \eqref{eq:wxprimegreater} and \eqref{eq:wxgreater} are less than or equal to \begin{align*}
&\mathbb{P}\left(|\zeta(s_0)M(s_0)-1|>\frac{1}{2}\right)+\mathbb{P}\left(|\zeta(s_0')M(s_0')-1|>\frac{1}{2}\right)\\
&\ll\mathbb{E}\left[|\zeta(s_0)M(s_0)-1|^2\right]+\mathbb{E}\left[|\zeta(s_0')M(s_0')-1|^2\right].
\end{align*}
We now introduce an analog of Proposition 4 of \cite{rands} as \cref{lma:zetaMestimate}.

\begin{lemma}\label{lma:zetaMestimate}
With the notation above, we have
\[\mathbb{E}[|\zeta(\sigma_0+i\tau)M(\sigma_0+i\tau)-1|^2]\ll(\log\log T)^{-K/K'}.\]
\end{lemma}

\begin{proof}
The proof follows the same procedure as in \cite{rands}, with slight modifications to accommodate for our different choice of \(W\). The main difference is in equation \eqref{eq:zetaproductparts}. The approximate functional equation for \(\zeta(\sigma_0+it)\) in the range \(T\leq t\leq 2T\) (e.g. Theorem 1.8 of \cite{ivic}) given by \(\zeta(\sigma_0+it)=\sum_{n\leq T}n^{-\sigma_0-it}+O(T^{-\frac{1}{2}})\) yields \begin{align*}
\int_T^{2T}\zeta(\sigma_0+it)M(\sigma_0+it)\,dt&=\int_T^{2T}\left[\sum_{n\leq T}\frac{1}{n^{\sigma_0+it}}+O(T^{-\frac{1}{2}})\right]\sum_m\frac{\mu(m)a(m)}{m^{\sigma_0+it}}\,dt\\
&=\sum_{n\leq T}\sum_m\frac{a(m)\mu(m)}{(mn)^{\sigma_0}}\int_T^{2T}(mn)^{-it}\,dt+O(T^{\frac{1}{2}+\frac{1}{K'}})\\
&=T+O(T^{\frac{1}{2}+\frac{1}{K'}}),
\end{align*} since the length of \(M\) being \(X=T^{1/(K'\log\log\log T)}\) implies \(M\ll T^{1/K'}\) and \(K'>2\). Thus we can expand the square to obtain \begin{equation}
\int_T^{2T}|1-\zeta(\sigma_0+it)M(\sigma_0+it)|^2\,dt=\int_T^{2T}|\zeta(\sigma_0+it)M(\sigma_0+it)|^2\,dt-T+O(T^{\frac{1}{2}+\frac{1}{K'}}).\label{eq:zetaMsquared}
\end{equation}
To evaluate this integral, we introduce Lemma 6 of \cite{selberg}.

\begin{lemma}\label{lma:selbergintegral}
Let \(h\) and \(k\) be non-negative integers, with \(h,k\leq T\). Then, for any \(1\geq\sigma>\frac{1}{2}\), \begin{align*}
\int_T^{2T}\left(\frac{h}{k}\right)^{it}|\zeta(\sigma+it)|^2\,dt&=\int_T^{2T}\left(\zeta(2\sigma)\left(\frac{(h,k)^2}{hk}\right)^\sigma+\left(\frac{t}{2\pi}\right)^{1-2\sigma}\zeta(2-2\sigma)\left(\frac{(h,k)^2}{hk}\right)^{1-\sigma}\right)dt\\
&+O(T^{1-\sigma+\epsilon}\min(h,k)).
\end{align*}
\end{lemma}

\noindent This lemma allows us to break down the integral in \eqref{eq:zetaMsquared} into the three summands \begin{align}
&\zeta(2\sigma_0)\left(\int_T^{2T}dt\right)\sum_{h,k}\frac{\mu(h)\mu(k)a(h)a(k)}{(hk)^{2\sigma_0}}(h,k)^{2\sigma_0},\label{eq:zetasummand1}\\
&\zeta(2-2\sigma_0)\left(\int_T^{2T}\left(\frac{t}{2\pi}\right)^{1-2\sigma_0}\,dt\right)\sum_{h,k}\frac{\mu(h)\mu(k)a(h)a(k)}{hk}(h,k)^{2-2\sigma_0},\label{eq:zetasummand2}\\
&O(T^{1-\sigma_0+\epsilon}\min(h,k))\sum_{h,k}\frac{\mu(h)\mu(k)a(h)a(k)}{(hk)^{\sigma_0}}.\label{eq:zetasummand3}
\end{align}
We split \(h\) and \(k\), so that \(h=h_1h_2\) and \(k=k_1k_2\) where \(h_1\) and \(k_1\) are composed of primes below \(Y\), and \(h_2\) and \(k_2\) of primes between \(Y\) and \(X\). We then write \(a(h)=a_1(h_1)a_2(h_2)\) for \begin{align*}
a_1(h_1)&=\begin{cases}1 & \text{if }\Omega(h_1)\leq 100\log\log T\text{ for }p\leq Y,\\
0 & \text{otherwise},\end{cases}\\
a_2(h_2)&=\begin{cases}1 & \text{if }\Omega(h_2)\leq 100\log\log\log T\text{ for }Y\leq p\leq X,\\
0 & \text{otherwise},\end{cases}
\end{align*} where \(\Omega(h)\) is the number of prime factors of \(h\) (with multiplicity), and write \(a(k)=a_1(k_1)a_2(k_2)\) similarly. The first summand \eqref{eq:zetasummand1} then factors as \begin{equation}
T\zeta(2\sigma_0)\left(\sum_{h_1,k_1}\frac{\mu(h_1)\mu(k_1)a_1(h_1)a_1(k_1)}{(h_1k_1)^{2\sigma_0}}(h_1,k_1)^{2\sigma_0}\right)\left(\sum_{h_2,k_2}\frac{\mu(h_2)\mu(k_2)a_2(h_2)a_2(k_2)}{(h_2k_2)^{2\sigma_0}}(h_2,k_2)^{2\sigma_0}\right).\label{eq:zetasummandfactored}
\end{equation}
Due to the multiplicativity of the factors in the sum, the first factor in \eqref{eq:zetasummandfactored} can be expressed as \[\sum_{\substack{h_1,k_1\\p\mid h_1k_1\,\Longrightarrow\, p\leq Y}}\frac{\mu(h_1)\mu(k_1)}{(h_1k_1)^{2\sigma_0}}(h_1,k_1)^{2\sigma_0}=\prod_{p\leq Y}\left(1-\frac{1}{p^{2\sigma_0}}-\frac{1}{p^{2\sigma_0}}+\frac{1}{p^{2\sigma_0}}\right)=\prod_{p\leq Y}\left(1-\frac{1}{p^{2\sigma_0}}\right),\]
together with an error term resulting from the cases where \(h_1\) and \(k_1\) have more than \(100\log\log T\) prime factors, which is \begin{align*}
&\ll\sum_{\substack{h_1,k_1\\p\mid h_1k_1\,\Longrightarrow\, p\leq Y\\ \Omega(h_1)>100\log\log T}}\frac{|\mu(h_1)\mu(k_1)|}{(h_1k_1)^{2\sigma_0}}(h_1,k_1)^{2\sigma_0}\\
&\ll e^{-100\log\log T}\sum_{\substack{h_1,k_1\\p\mid h_1k_1\,\Longrightarrow\, p\leq Y}}\frac{|\mu(h_1)\mu(k_1)|}{(h_1k_1)^{2\sigma_0}}(h_1,k_1)^{2\sigma_0}e^{\Omega(h_1)}
\end{align*}
using Rankin's trick. Here we have used symmetry to assume that \(h_1\) has more than \(100\log\log T\) factors. Then once again, due to multiplicativity, this is \begin{align*}
&\ll(\log T)^{-100}\prod_{p\leq Y}\left(1+\frac{e}{p^{2\sigma_0}}+\frac{1}{p^{2\sigma_0}}+\frac{e}{p^{2\sigma_0}}\right)=(\log T)^{-100}\prod_{p\leq Y}\left(1+\frac{1+2e}{p^{2\sigma_0}}\right)\\
&=(\log T)^{-100}\exp\left[\sum_{p\leq Y}\left(\frac{1+2e}{p^{2\sigma_0}}\right)+O\left(\frac{1}{p^{4\sigma_0}}\right)\right]\\
&\ll(\log T)^{-100}\exp[(1+2e)\log\log Y]\ll(\log T)^{-90}.
\end{align*} Similarly, the second factor in \eqref{eq:zetasummandfactored} is \[\prod_{Y<p\leq X}\left(1-\frac{1}{p^{2\sigma_0}}\right)+O((\log\log T)^{-90}).\] By inserting these factors into \eqref{eq:zetasummandfactored}, the first summand \eqref{eq:zetasummand1} is \[\sim T\zeta(2\sigma_0)\prod_{p\leq X}\left(1-\frac{1}{p^{2\sigma_0}}\right)=T\prod_{p>X}\left(1-\frac{1}{p^{2\sigma_0}}\right)^{-1}.\] We then use the prime number theorem with classical error (e.g. Theorem 6.9 of \cite{montgomery}), which states that the prime-counting function \(\pi(X)\) can be written as \[\pi(X)=\int_2^X\frac{1}{\log y}\,dy+O(xe^{-c\sqrt{\log X}}).\] This allows us to write the product as \[\prod_{p>X}\left(1-\frac{1}{p^{2\sigma_0}}\right)^{-1}=\exp\left(\int_X^\infty\frac{1}{\log t}\frac{1}{t^{2\sigma_0}}\,dt+o(1)\right),\] where the integral can be evaluated using integration by parts to obtain \begin{align}
\left.\frac{t^{1-2\sigma_0}}{(1-2\sigma_0)\log t}\right\rvert_X^\infty+\int_X^\infty\frac{1}{(1-2\sigma_0)t^{2\sigma_0}(\log t)^2}\,dt&\ll\frac{X^{1-2\sigma_0}}{(2\sigma_0-1)\log X}+\frac{X^{1-2\sigma_0}}{(1-2\sigma_0)^2(\log X)^2}\label{eq:zetaproductparts}\\
&\ll \frac{X^{-W/\log T}}{\frac{W}{\log T}\log X}\nonumber\\
&=\frac{\exp[-W/(K'\log\log\log T)]}{W/(K'\log\log\log T)}\nonumber\\
&=\frac{\exp(-K\log\log\log T/K')}{K\log\log\log T/K'}\nonumber\\
&\ll(\log\log T)^{-K/K'}.\nonumber
\end{align}
By similar reasoning, the second summand \eqref{eq:zetasummand2} is \begin{align*}
&\sim\left(\int_T^{2T}\left(\frac{t}{2\pi}\right)^{1-2\sigma_0}\,dt\right)\zeta(2-2\sigma_0)\prod_{p\leq X}\left(1-\frac{2}{p}+\frac{1}{p^{2\sigma_0}}\right)=o(T).
\end{align*}
Thus since the error term \eqref{eq:zetasummand3} can be written as \[O\left(T^{1/2+\epsilon}\sum_{h,k}\frac{1}{(hk)^{\sigma_0}}\min(h,k)\right)=O\left(T^{1/2+\epsilon}\right),\] which is \(o(T)\), the desired expectation estimate results from the contribution by the first summand \eqref{eq:zetasummand1}.
\end{proof}

\noindent Finishing the proof of \cref{prop:mollifying}, observe that by \cref{lma:zetaMestimate}, both \(\mathbb{E}\left[|\zeta(s_0)M(s_0)-1|^2\right]\) and \(\mathbb{E}\left[|\zeta(s_0')M(s_0')-1|^2\right]\) are \(\ll (\log\log T)^{-K/K'}\), and our result follows, since \(K>K'\).
\end{proof}

\subsection{Proof of \cref{prop:approximatingthemollifier}} \begin{proof}
Similar to the way we proved \cref{prop:movingoffaxis} and \cref{prop:mollifying}, we have
\[d_{\mathcal{D}}(\mathbf{X},\mathbf{Y})=\sup_{f\in\mathcal{L}}|\mathbb{E}[f(\mathbf{X})]-\mathbb{E}[f(\mathbf{Y})]|\leq\sup_{f\in\mathcal{L}}\mathbb{E}[|f(\mathbf{X})-f(\mathbf{Y})|].\]
Next set \(P_1(s)=\re\sum_{p\leq Y}(1/p^s)\) and \(P_2(s)=\re\sum_{Y\leq p\leq X}(1/p^s)\), so that \(P(s)=P_1(s)+P_2(s)\). Recall the definitions of \(s_0,s_0'\) given in equation \eqref{eq:sigma0}. We then evaluate the expectation by splitting it into separate cases, so that the last term is
\begin{align}
&\ll\mathbb{E}[\|\mathbf{X}-\mathbf{Y}\|_1\cdot\mathbbm{1}(|P_1(r)|\leq\log\log T,|P_2(r)|\leq\log\log\log T,r\in\{s_0,s_0'\})]\label{eq:xyleq}\\
&\phantom{\ll}+\sum_{r\in\{s_0,s_0'\}}[\mathbb{P}(|P_1(r)|>\log\log T)+\mathbb{P}(|P_2(r)|>\log\log\log T)],\label{eq:xygreater}
\end{align}
where we use the Lipschitz property in the definition of \(\mathcal{L}\) for \eqref{eq:xyleq} and the boundedness property for \eqref{eq:xygreater}.
First, by considering \eqref{eq:xyleq}, we see that \begin{align*}
&\mathbb{E}[\|\mathbf{X}-\mathbf{Y}\|_1\cdot\mathbbm{1}(|P_1(r)|\leq\log\log T,|P_2(r)|\leq\log\log\log T,r\in\{s_0,s_0'\})]\\
&=\mathbb{E}[|\mathcal{X}-\mathcal{Y}|\cdot\mathbbm{1}(|P_1(s_0)|\leq\log\log T,|P_2(s_0)|\leq\log\log\log T)\\
&\phantom{=}+|\mathcal{X}'-\mathcal{Y}'|\cdot\mathbbm{1}(|P_1(s_0')|\leq\log\log T,|P_2(s_0')|\leq\log\log\log T)]\\
&=\frac{1}{\mathfrak{s}}\mathbb{E}[|\log|M^{-1}(s_0)|-\re\mathcal{P}(s_0)|\cdot\mathbbm{1}(|P_1(s_0)|\leq\log\log T,|P_2(s_0)|\leq\log\log\log T)]\\
&\phantom{=}+\frac{1}{\mathfrak{s}}\mathbb{E}[|\log|M^{-1}(s_0')|-\re\mathcal{P}(s_0')|\cdot\mathbbm{1}(|P_1(s_0')|\leq\log\log T,|P_2(s_0')|\leq\log\log\log T)].
\end{align*}
To finish our bound on \eqref{eq:xyleq} we calculate the necessary expectations on the given events with the following lemma, which shows that both summands are \(\ll (\log\log T)^{-80}\). Here we take a different approach than Proposition 3 of \cite{rands} in order to clearly quantify the associated error with this step of the proof of \cref{thm:selbergcltrate}.

\begin{lemma}\label{lma:logM-1-reP}
With notation as above, we have \[|\log|M^{-1}(s)|-\re\mathcal{P}(s)|\ll(\log\log T)^{-80}\] whenever the events \(\{|\mathcal{P}_1(s)|\leq\log\log T\}\), \(\{|\mathcal{P}_2(s)|\leq\log\log\log T\}\), \(\{|P_1(s)|\leq\log\log T\}\), and \(\{|P_2(s)|\leq\log\log\log T\}\) hold for \begin{align*}
\mathcal{P}_1(s)&=\sum_{2\leq n\leq Y}\frac{\Lambda(n)}{n^s\log n}&\mathcal{P}_2(s)&=\sum_{Y< n\leq X}\frac{\Lambda(n)}{n^s\log n}\\
P_1(s)&=\re\sum_{p\leq Y}\frac{1}{p^s} & P_2(s)&=\re\sum_{Y<p\leq X}\frac{1}{p^s}
\end{align*} (so that \(\mathcal{P}(s)=\mathcal{P}_1(s)+\mathcal{P}_2(s)\), \(P(s)=P_1(s)+P_2(s)\)).
\end{lemma}

\begin{proof}
Let \begin{align*}
a_1(n)&=\begin{cases}1 & \text{if }p\mid n\Rightarrow p\leq Y,\\
& \text{and }\Omega(n)\leq 100\log\log T\text{ for }p\leq Y,\\
0 & \text{otherwise}.\end{cases}\\
\tilde{a}_1(n)&=\begin{cases}1 & \text{if }p\mid n\Rightarrow p\leq Y,\\
& \text{and }\Omega(n)> 100\log\log T\text{ for }p\leq Y,\\
0 & \text{otherwise}.\end{cases}
\end{align*} and set \[M_1(s)=\sum_n\frac{\mu(n)a_1(n)}{n^s}\qquad\text{and}\qquad\mathcal{E}_1(s)=\sum_n\frac{\mu(n)\tilde{a}_1(n)}{n^s}.\]  Then using the Taylor series for \(\log(1-x)\), we get \begin{align}
\exp(-\mathcal{P}_1(s))&=\exp\left(-\sum_{k\geq 1}\sum_{p\leq Y}\frac{(p^{-s})^k}{k}\right)=\exp\left(\sum_{p\leq Y}\log(1-p^{-s})\right)\nonumber\\
&=\prod_{p\leq Y}(1-p^{-s})=M_1(s)+\mathcal{E}_1(s).\label{eq:MplusE}
\end{align} We rewrite \(\mathcal{E}_1(s)\), as in Lemma 23 of \cite{fyodorov}, \begin{align*}
\mathcal{E}_1(s)&=\sum_n\frac{\mu(n)\tilde{a}_1(n)}{n^s}=\sum_{\ell>100\log\log T}(-1)^\ell\left(\sum_{p_1<\cdots<p_\ell}\frac{1}{(p_1\cdots p_\ell)^s}\right)\\
&=\sum_{\ell>100\log\log T}(-1)^\ell\left(\sum_{\substack{m_1,\ldots,m_k,\ldots \\ m_1+2m_2+\cdots+km_k+(k+1)m_{k+1}+\cdots=\ell}}\prod_{1\leq j}\frac{(-P_1(js))^{m_j}}{m_j!j^{m_j}}\right)\\
&\ll\sum_{m_1,\ldots,m_k,\ldots}e^{(-100\log\log T+m_1+2m_2+\cdots+km_k+\cdots)}\prod_{1\leq j}\frac{|P_1(js)|^{m_j}}{m_j!j^{m_j}}\\
&=e^{(-100\log\log T)}e^{\sum_{1\leq j}e^j|P_1(js)|/j}.
\end{align*}
By assumption, \(|P_1(s)|\leq\log\log T\). In addition, since \(|P_1(2s)|\ll\log\log T\) by Mertens's Theorem (as in \cref{lma:covariance}) and \(|P_1(\ell s)|\leq 10^{-\ell}\) for \(\ell\geq 3\) by summability, we obtain a bound of \((\log T)^{-90}\) for \(\mathcal{E}_1(s)\). We also have \(|\mathcal{P}_1(s)|\leq\log\log T\) by assumption, and this gives us a lower bound of \((\log T)^{-1}\) for \(M_1(s)\) via \eqref{eq:MplusE}, since \begin{align*}
(\log T)^{-1}\leq\exp(-\mathcal{P}_1(s))=M_1(s)+\mathcal{E}_1(s).
\end{align*} Inserting these bounds into \eqref{eq:MplusE} and taking the \(\log\) of both sides, we see that \begin{align*}
\re\mathcal{P}_1(s)&=-\log|M_1(s)+\mathcal{E}_1(s)|\\
&=\log\left|M_1(s)\left(1+\frac{\mathcal{E}_1(s)}{M_1(s)}\right)\right|^{-1}\\
&=\log|M_1^{-1}(s)|+O((\log T)^{-80}).
\end{align*}

\noindent We can define \(M_2\) similarly for \(\mathcal{P}_2(s)\), so that \(M(s)=M_1(s)M_2(s)\) and \(\mathcal{E}_2(s)\ll(\log\log T)^{-90}\) when \(|P_2(s)|\leq\log\log\log T\). We can also obtain a lower bound of \((\log\log T)^{-1}\) for \(M_2(s)\), giving us \[\re\mathcal{P}_2(s)=\log|M_2^{-1}(s)|+O((\log\log T)^{-80}).\] Since this error is worse than our error for \(\mathcal{P}_1(s)\), it becomes our error estimate for \(\mathcal{P}(s)\).
\end{proof}

Turning our attention to \eqref{eq:xygreater}, we calculate the required probabilities using the moments of \(P_1(s)\) and \(P_2(s)\) given by the following lemma. Note that we introduce Lemma 3.4 of \cite{maxzeta} as \cref{lma:Pmoments} to calculate the moments of \(P_1(s)\), instead of the process used in Proposition 2 of \cite{rands}, in order to compare moments of linear combinations of \(P_1(s)\) when proving \cref{prop:comparingmoments} in equation \eqref{eq:fourierdifference}.

\begin{lemma}\label{lma:Pmoments}
For any non-negative integer \(k\) we have \begin{align*}
&\mathbb{E}\left[\left(\re\left(u\sum_{p\leq Y}\frac{1}{p^{s_0}}+u'\sum_{p\leq Y}\frac{1}{p^{s_0'}}\right)\right)^{2k}\right]\\
&=2^{-k}\frac{(2k)!}{k!}[V(u,u')^{k}+O_k(V(u,u')^{k-2})]+O\left(\frac{Y^{4k}}{T}\right),
\end{align*}
where \(O_k\) indicates that the implied constant depends on \(k\) and \[V(u,u')=\left(\frac{1}{2}\sum_{p\leq Y} \frac{1}{p^{2\sigma_0}}(u^2+(u')^2+2uu'\cos(|h-h'|\log p))\right).\]
In particular, the expression is \[\ll 2^{-k}\frac{(2k)!}{k!}\left(\frac{1}{2}\sum_{p\leq Y} \frac{1}{p}(u^2+(u')^2+2uu'\cos(|h-h'|\log p))\right)^{k}+O\left(\frac{Y^{4k}}{T}\right).\]
If \(k\) is odd, we have \[\mathbb{E}\left[\left(\re\left(u\sum_{p\leq Y}\frac{1}{p^{s_0}}+u'\sum_{p\leq Y}\frac{1}{p^{s_0'}}\right)\right)^k\right]=O\left(\frac{Y^{2k}}{T}\right).\]
\end{lemma}

\begin{proof}
Fix first an integer \(k\). Rearranging terms gives \begin{align}
&\mathbb{E}\left[\left(\re\left(u\sum_{p\leq Y}\frac{1}{p^{\sigma_0+i(\tau+h)}}+u'\sum_{p\leq Y}\frac{1}{p^{\sigma_0+i(\tau+h')}}\right)\right)^k\right]\nonumber\\
&=\frac{1}{T}\int_T^{2T}2^{-k}\left[\sum_{p\leq Y}\left(\frac{up^{-ih}+u'p^{-ih'}}{p^{\sigma_0}}p^{-it}+\frac{up^{ih}+u'p^{ih'}}{p^{\sigma_0}}p^{it}\right)\right]^kdt.\label{eq:Pmoments}
\end{align}
Set \[a(p)=\frac{up^{-ih}+u'p^{-ih'}}{p^{\sigma_0}},\qquad b(p)=\frac{up^{ih}+u'p^{ih'}}{p^{\sigma_0}}\] and \[a(n)=\prod_ja(p_j)^{\alpha_j},\qquad b(n)=\prod_jb(p_j)^{\alpha_j},\qquad g(p^\alpha)=\frac{1}{\alpha !}\] where \(n=p_1^{\alpha_1}\cdots p_r^{\alpha_r}\). Then, \eqref{eq:Pmoments} becomes \begin{align*}
&\frac{1}{T}\int_T^{2T}2^{-k}\left(\sum_{p\leq Y}\left(a(p)p^{-it}+b(p)p^{it}\right)\right)^kdt\\
&=\frac{1}{T}\int_T^{2T}2^{-k}\sum_{\ell=0}^k\begin{pmatrix}
k \\
\ell
\end{pmatrix}\left(\ell !\sum_{\Omega(m)=\ell}a(m)g(m)m^{-it}\right)\left((k-\ell)!\sum_{\Omega(n)=k-\ell}b(n)g(n)n^{it}\right)dt.
\end{align*} Since expanding the integral gives \begin{align*}
&\int_T^{2T}\sum_{\ell=0}^k\left(\sum_{\Omega(m)=\ell}a(m)m^{-it}\right)\left(\sum_{\Omega(n)=k-\ell}b(n)n^{it}\right)dt\\
&=\sum_{\ell=0}^k\sum_{\substack{\Omega(m)=\ell \\ \Omega(n)=k-\ell}}a(m)b(n)\int_T^{2T}\left(\frac{n}{m}\right)^{it}dt\\
&=\sum_{\ell=0}^k\left(T\sum_{\Omega(n)=2\ell}a(n)b(n)+O\left(\sum_{\substack{\Omega(m)=\ell,\,\,\Omega(n)=k-\ell \\ m\ne n}}\frac{|a(m)b(n)|}{|\log(m/n)|}\right)\right),
\end{align*} where the remainder term in the sum is \begin{align*}
\ll\sum_{\Omega(m)=\ell}|a(m)|^2\sum_{\substack{\Omega(n)=k-\ell \\ n\ne m}}\frac{1}{|\log(m/n)|}+\sum_{\Omega(n)=k-\ell}|b(n)|^2\sum_{\substack{\Omega(m)=\ell \\ m\ne n}}\frac{1}{|\log(m/n)|}\ll Y\log Y,
\end{align*} we see that the main term resulting from the evaluation of the expectation above arises when \(m=n\), making \(\ell=\Omega(m)=\Omega(n)=k-\ell\). Thus if \(k\) is odd there are no main terms and just the remainder

\begin{align*}
O_{u,u'}\left(\frac{Y^k\log(Y^k)}{T}2^{-k}\sum_{\ell=0}^k\begin{pmatrix}
k \\
\ell
\end{pmatrix}\left(\pi(Y)^\ell+\pi(Y)^{k-\ell}\right)\right)&=O_{u,u'}\left(\frac{Y^k\log(Y^k)}{T}2^{-k}\left(\frac{Y}{\log Y}\right)^{k}\right)\\
&=O_{u,u'}\left(\frac{Y^{2k}}{T}\right)
\end{align*} where \(\pi(Y)\sim Y/\log Y\) denotes the number of primes below \(Y\).\\
Next by replacing \(k\) with \(2k\) in \eqref{eq:Pmoments} to compute the even moments, we get \begin{equation}
2^{-2k}(2k)!\sum_{\Omega(n)=k}a(n)b(n)g(n)^2+O\left(\frac{Y^{4k}}{T}\right).\label{eq:replacewith2k}
\end{equation} We then split the sum into cases where \(n\) is square-free and is not square-free so that
\[k!\sum_{\Omega(n)=k}a(n)b(n)g(n)^2=k!\sum_{\substack{\Omega(n)=k \\ n\text{ is square free}}}a(n)b(n)+k!\sum_{\substack{\Omega(n)=k \\ n\text{ is not square free}}}a(n)b(n)g(n)^2.\]
When \(n\) is square-free, we have
\begin{align*}
k!\sum_{\substack{\Omega(n)=k \\ n\text{ is square free}}}a(n)b(n)&=k!\sum_{p_1<\cdots<p_k}a(p_1\cdots p_k)b(p_1\cdots p_k)\\
&=\sum_{p_1,\ldots,p_k}a(p_1\cdots p_k)b(p_1\cdots p_k)\\
&=\left(\sum_{p\leq Y}a(p)b(p)\right)^k.
\end{align*}
When \(n\) is not square-free, the worst possibility is when \(n\) has a single prime factor with power \(\alpha_k>1\), in which case we have
\begin{align*}
k!\sum_{\substack{\Omega(n)=k \\ n\text{ is not square free}}}a(n)b(n)g(n)^2&=\frac{1}{(\alpha_k)^2}\sum_{\Omega(n)=k}a(n)b(n)\frac{k!}{\alpha_1!\cdots\alpha_{k-2}!}\\
&=O_k\left(\sum_{p_1,\ldots,p_{k-2}}a(p_1\cdots p_{k-2})b(p_1\cdots p_{k-2})\right)\\
&=O_k\left(\sum_{p\leq Y} a(p)b(p)\right)^{k-2}.
\end{align*}
Thus \eqref{eq:replacewith2k} is equal to
\begin{align}
2^{-2k}\frac{(2k)!}{k!}\left[\left(\sum_{p\leq Y} a(p)b(p)\right)^k+O_k\left(\sum_{p\leq Y} a(p)b(p)\right)^{k-2}\right]+O\left(\frac{Y^{4k}}{T}\right).\label{eq:evenmoments}
\end{align} Then, since \begin{align*}
a(p)b(p)&=\frac{1}{p^{2\sigma_0}}(up^{-ih}+u'p^{-ih'})(up^{ih}+u'p^{ih'})\\
&=\frac{1}{p^{2\sigma_0}}(u^2+(u')^2+2uu'\re p^{i(h'-h)}),
\end{align*} equation \eqref{eq:evenmoments} becomes \[2^{-k}\frac{(2k)!}{k!}[V(u,u')^k+O_k(V(u,u'))^{k-2}]+O\left(\frac{Y^{4k}}{T}\right),\]
which is \begin{align}
&\ll 2^{-k}\frac{(2k)!}{k!}[V(u,u')^k]+O\left(\frac{Y^{4k}}{T}\right)\nonumber \\
&=2^{-k}\frac{(2k)!}{k!}\left(\frac{1}{2}\sum_{p\leq Y} \frac{1}{p}(u^2+(u')^2+2uu'\cos(|h-h'|\log p))\right)^{k}+O\left(\frac{Y^{4k}}{T}\right).\label{eq:evenmomentsestimate}
\end{align}
\end{proof}

Without the error terms, \cref{lma:Pmoments} suggests that the moment-generating function for \((P_1(s_0),P_1(s_0'))\) matches that of a Gaussian vector with mean \(0\) and covariance matrix \[\frac{1}{2}\begin{blmatrix}
\displaystyle\sum_{p\leq Y}\frac{1}{p} & \displaystyle\sum_{p\leq Y}\frac{\cos(|h-h'|\log p)}{p}\\
\displaystyle\sum_{p\leq Y}\frac{\cos(|h-h'|\log p)}{p} & \displaystyle\sum_{p\leq Y}\frac{1}{p}
\end{blmatrix}.\]

The sums in this covariance matrix can be evaluated and normalized to approximate the entries given in \cref{thm:selbergcltrate}. This is done in the following lemma, similar to Lemma 2.1 of \cite{maxzetawalks}.

\begin{lemma}\label{lma:covariance}
We have
\[\sum_{p\leq Y}\frac{1}{p}=\log\log T+O(1),\] and when \(|h-h'|\sim (\log T)^{-\alpha}\) for \(\alpha\in(0,1)\) we have
\[\sum_{p\leq Y}\frac{\cos(|h-h'|\log p)}{p}=\alpha\log\log T+O(1).\]
\end{lemma}

\begin{proof}
Since \(\log\log Y=\log\log T-\log\log((\log T)^{K'})\), we evaluate the first sum using Mertens's theorem (e.g. see \cite{iwaniec} (2.15)), which states that \[\sum_{p\leq x}\frac{1}{p}=\log\log x+O(1).\] To evaluate the second sum, we proceed similarly to Lemmas 3.3 and 3.4 of Bourgade \cite{mesoscopic} and use the prime number theorem together with summation by parts to get \begin{align*}
&\sum_{p\leq Y}\frac{\cos(|h-h'|\log p)}{p}\\
&=\int_2^Y\frac{\cos(|h-h'|\log t)}{t\log t}\,dt+O\left(e^{-c\sqrt{\log 2}}+\int_2^Y\tfrac{\cos(|h-h'|\log t)+|h-h'|\sin(|h-h'|\log t)}{t}e^{-c\sqrt{\log t}}\,dt\right).
\end{align*} Taking the change of variable \(w=|h-h'|\log t\) gives \[\int_2^Y\frac{|h-h'|\sin(|h-h'|\log t)}{t}e^{-c\sqrt{\log t}}\,dt=\int_{|h-h'|\log 2}^{|h-h'|\log Y}\sin we^{-c\sqrt{w/|h-h'|}}\,dw\ll e^{-c\sqrt{\log 2}},\] and the integral with cosine in the error term is evaluated similarly. To evaluate the last remaining integral we split it into cases where the argument of cosine is smaller than \(1\) and where it is greater than or equal to \(1\), obtaining \[\int_2^Y\frac{\cos(|h-h'|\log t)}{t\log t}\,dt=\int_2^{e^{|h-h'|^{-1}}}\frac{\cos(|h-h'|\log t)}{t\log t}\,dt+\int_{e^{|h-h'|^{-1}}}^Y\frac{\cos(|h-h'|\log t)}{t\log t}\,dt.\] For the first integral, we apply the estimate \(\cos(|h-h'|v)=1+O(|h-h'|^2v^2)\) for \(v=\log t\) on the range \(\log 2\leq v\leq |h-h'|^{-1}\) to get \begin{align*}
\int_2^{e^{|h-h'|^{-1}}}\frac{\cos(|h-h'|\log t)}{t\log t}\,dt&=\int_{\log 2}^{|h-h'|^{-1}}\frac{\cos(|h-h'|v)}{v}\,dv\\
&=\int_{\log 2}^{|h-h'|^{-1}}\frac{1}{v}\,dv+O\left(|h-h'|^2\int_{\log 2}^{|h-h'|^{-1}}v\,dv\right)\\
&=\log |h-h'|^{-1}-\log\log 2+O(1).
\end{align*}
For the second integral, we integrate by parts, which gives
\[\int_{|h-h'|^{-1}}^{\log Y}\frac{\cos(|h-h'|v)}{v}\,dv=\left.\frac{\sin(|h-h'|v)}{|h-h'|v}\right\rvert_{|h-h'|^{-1}}^{\log Y}+\int_{|h-h'|^{-1}}^{\log Y}\frac{\sin(|h-h'|v)}{|h-h'|v^2}\,dv=O(1),\] since both terms are \(O(1)\).
Therefore \begin{align*}
\sum_{p\leq Y}\frac{\cos(|h-h'|\log p)}{p}&=\log |h-h'|^{-1}+O(1)=\alpha\log\log T+O(1).
\end{align*}
\end{proof}

We now use \cref{lma:Pmoments} to bound the probabilities of large values of \(|P_1(s)|\) and \(|P_2(s)|\) in \cref{lma:largevaluesofP} and finish the proof. 

\begin{lemma}\label{lma:largevaluesofP}
With notation as above, we have \begin{align*}
\mathbb{P}(|P_1(s)|>\log\log T)&\ll\frac{1}{\log T}\\
\mathbb{P}(|P_2(s)|>\log\log\log T)&\ll\frac{1}{\log\log T}.
\end{align*}
\end{lemma}

\begin{proof}
By the upper bound \eqref{eq:evenmomentsestimate}, in particular for \(u=1,u'=0\), we have \[\mathbb{E}[|P_1(s)|^{2k}]\ll\frac{(2k)!}{k!2^k}\left(\sqrt{\tfrac{1}{2}\log\log T}\right)^{2k}\] when \(k\) is small enough to handle the error term of \(O(Y^{4k}/T)\). Therefore, we can set \(k=\lfloor\log\log T\rfloor\) due to our choice of \(Y\). The way we obtained the moments of \((P_1(s_0),P_1(s_0'))\) in \cref{lma:Pmoments} can also be used to show that \begin{equation}
\mathbb{E}[|P_2(s)|^{2k}]\ll\frac{(2k)!}{k!2^k}\left(\sqrt{\tfrac{1}{2}\log\log\log T}\right)^{2k}+O\left(\frac{X^{4k}}{T}\right).\label{eq:P2moments}
\end{equation} We then use Markov's inequality to obtain \begin{align*}
\mathbb{P}(|P_1(s)|>\log\log T)&\leq\frac{1}{(\log\log T)^{2k}}\mathbb{E}[|P_1(s)|^{2k}]\\
&\ll\frac{1}{(\log\log T)^{2k}}\left[\frac{(2k)!}{k!2^k}\left(\sqrt{\tfrac{1}{2}\log\log T}\right)^{2k}\right]\\
&\ll\frac{1}{(\log\log T)^{2k}}\left[\frac{(2k)^{2k}}{k^k2^k}\frac{\exp(-2k)}{\exp(-k)}\left(\sqrt{\tfrac{1}{2}\log\log T}\right)^{2k}\right]\\
&=\exp\left[-k\log\frac{(\log\log T)^2}{\log\log T}+k\log k-k\right].
\end{align*}
Evaluating at \(k=\lfloor\log\log T\rfloor\), the last line becomes \(\exp(-\log\log T)=1/\log T\), and the result follows for \(P_1(s)\). The proof for \(P_2(s)\) is similar.
\end{proof}

By \cref{lma:largevaluesofP}, \eqref{eq:xygreater} is \(\ll 1/\log\log T\), and thus the significant contribution towards the error in \cref{prop:approximatingthemollifier} comes from \eqref{eq:xyleq}.
\end{proof}

\subsection{Proof of \cref{prop:discardingprimes}} \begin{proof}
Similar to the way we proved \crefrange{prop:movingoffaxis}{prop:approximatingthemollifier}, by taking advantage of the Lipschitz nature of the Dudley distance, we obtain \begin{align*}
&d_{\mathcal{D}}\left[\mathbf{Y},\left(\frac{P(s_0)}{\mathfrak{s}},\frac{P(s_0')}{\mathfrak{s}}\right)\right]\\
&\leq\mathbb{E}\left[\left\|\mathbf{Y}-\left(\frac{P(s_0)}{\mathfrak{s}},\frac{P(s_0')}{\mathfrak{s}}\right)\right\|_1\right]\\
&=\mathbb{E}\left[\left|\mathcal{Y}-\frac{P(s_0)}{\mathfrak{s}}\right|+\left|\mathcal{Y}'-\frac{P(s_0')}{\mathfrak{s}}\right|\right]\\
&=\mathbb{E}\left[\left|\mathcal{Y}-\frac{P(s_0)}{\mathfrak{s}}\right|\right]+\mathbb{E}\left[\left|\mathcal{Y}'-\frac{P(s_0')}{\mathfrak{s}}\right|\right].
\end{align*}
Computing the difference, we get, for \(s=s_0\) or \(s=s_0'\), \begin{equation}
\frac{1}{\mathfrak{s}}|\re\mathcal{P}(s)-P(s)|\leq\frac{1}{\mathfrak{s}}\left|\sum_{2\leq p^2\leq X}\frac{1}{2p^{2s}}\right|+\frac{1}{\mathfrak{s}}\left|\sum_{\substack{2\leq p^k\leq X,\\k\geq 3}}\frac{\log p}{p^{ks}(k\log p)}\right|.\label{eq:reP-P}
\end{equation} Then since \[\left|\sum_{\substack{2\leq p^k\leq X,\\k\geq 3}}\frac{\log p}{p^{ks}(k\log p)}\right|\leq\sum_{\substack{2\leq p^k\leq X,\\k\geq 3}}\frac{1}{3p^{k\sigma_0}}=O(1)\]
and \begin{align*}
\mathbb{E}\left[\left|\sum_{2\leq p^2\leq X}\frac{1}{2p^{2(\sigma_0+i\tau)}}\right|\right]&\leq\mathbb{E}\left[\left|\sum_{2\leq p^2\leq X}\frac{1}{2p^{2(\sigma_0+i\tau)}}\right|^2\right]^{1/2}\\
&=\frac{1}{T}\int_T^{2T}\left|\sum_{2\leq p^2\leq X}\frac{1}{2p^{2(\sigma_0+it)}}\right|^2dt\\
&=\frac{1}{T}\cdot\frac{1}{4}\sum_{p_1,p_2\leq\sqrt{X}}\int_T^{2T}\frac{1}{p_1^{2(\sigma_0+it)}p_2^{2(\sigma_0-it)}}\,dt\\
&\ll \sum_{p\leq\sqrt{X}}\frac{1}{p^{4\sigma_0}}+\sum_{\substack{p_1,p_2\leq\sqrt{X}\\p_1\ne p_2}}\frac{1}{p_1^{2\sigma_0}p_2^{2\sigma_0}}\sqrt{p_1p_2}=O(1),
\end{align*}
we see that \eqref{eq:reP-P} is \(\ll 1/\mathfrak{s}\).
\end{proof}

\subsection{Proof of \cref{prop:truncatingprimesum}} \begin{proof}
We have
\begin{align*}
d_{\mathcal{D}}\left[\left(\frac{P(s_0)}{\mathfrak{s}},\frac{P(s_0')}{\mathfrak{s}}\right),\left(\frac{P_1(s_0)}{\mathfrak{s}},\frac{P_1(s_0')}{\mathfrak{s}}\right)\right]\leq\mathbb{E}\left[\left(\frac{P_2(s_0)}{\mathfrak{s}}\right)^2+\left(\frac{P_2(s_0')}{\mathfrak{s}}\right)^2\right]^{1/2}.
\end{align*}
By the upper bound \eqref{eq:P2moments} for the moments of \(P_2(s)\), the expectation above is less than or equal to \(\frac{1}{\mathfrak{s}}(\sqrt{\log\log\log T}+O(1))\). Similarly, \begin{align*}
d_{\mathcal{D}}\left[\left(\frac{P_1(s_0)}{\mathfrak{s}},\frac{P_1(s_0')}{\mathfrak{s}}\right),\left(\frac{P_1(s_0)}{\tilde{\mathfrak{s}}},\frac{P_1(s_0')}{\tilde{\mathfrak{s}}}\right)\right]\leq\mathbb{E}\left[\left(\frac{P_1(s_0)(\tilde{\mathfrak{s}}-\mathfrak{s})}{\tilde{\mathfrak{s}}\mathfrak{s}}\right)^2+\left(\frac{P_1(s_0')(\tilde{\mathfrak{s}}-\mathfrak{s})}{\tilde{\mathfrak{s}}\mathfrak{s}}\right)^2\right]^{1/2},
\end{align*}
where \(\tilde{\mathfrak{s}}-\mathfrak{s}=O(1)\) by Mertens's Theorem. Thus the expectation is of order \(1/\sqrt{\log\log T}\) due to the upper bound \eqref{eq:evenmomentsestimate} for the moments of \(P_1(s)\).
\end{proof}

\subsection{Proof of \cref{prop:comparingmoments}} \begin{proof}

We obtain the error in approximating \((P_1(s_0),P_1(s_0'))\) with a Gaussian vector by comparing their Fourier transforms using Lemma 3.11 of \cite{maxzeta}, which we introduce as \cref{lma:fourierestimate},
\begin{lemma}\label{lma:fourierestimate}
Let \(\mu\) and \(\nu\) be two probability measures on \(\mathbb{R}^2\) with Fourier transform \(\hat{\mu}\) and \(\hat{\nu}\). There exists constant \(c>0\) such that for any function \(f:\mathbb{R}^2\to\mathbb{R}\) with Lipschitz constant \(1\) and for any \(R,F>0\), \[\left|\int_{\mathbb{R}^2}f\,d \mu-\int_{\mathbb{R}^2}f\,d \nu\right|\leq\frac{1}{F}+\|f\|_\infty\left\{(RF)^2\|(\hat{\mu}-\hat{\nu})\mathbbm{1}_{(-F,F)^2}\|_\infty+\mu(((-R,R)^2)^c)+\nu(((-R,R)^2)^c)\right\}.\]
\end{lemma}

To use \cref{lma:fourierestimate}, we first truncate the Fourier transform of \((P_1(s_0),P_1(s_0'))\) for \(\xi,\xi'\in\mathbb{R}\) by writing \begin{align*}
&\mathbb{E}\left[\exp\left(\frac{i\xi P_1(s_0)+i\xi' P_1(s_0')}{\tilde{\mathfrak{s}}}\right)\cdot\mathbbm{1}(|P_1(r)|\leq\log\log T,r\in\{s_0,s_0'\})\right]\\
&=\sum_{n\leq N}\frac{1}{n!}i^n\mathbb{E}\left[\left(\frac{\xi P_1(s_0)+\xi'P_1(s_0')}{\tilde{\mathfrak{s}}}\right)^n\cdot\mathbbm{1}(|P_1(r)|\leq\log\log T,r\in\{s_0,s_0'\})\right]\\
&\phantom{=}+\sum_{n> N}\frac{1}{n!}i^n\mathbb{E}\left[\left(\frac{\xi P_1(s_0)+\xi'P_1(s_0')}{\tilde{\mathfrak{s}}}\right)^n\cdot\mathbbm{1}(|P_1(r)|\leq\log\log T,r\in\{s_0,s_0'\})\right]
\end{align*}
where \begin{align*}
&\sum_{n> N}\frac{1}{n!}i^n\mathbb{E}\left[\left(\frac{\xi P_1(s_0)+\xi' P_1(s_0')}{\tilde{\mathfrak{s}}}\right)^n\cdot\mathbbm{1}(|P_1(r)|\leq\log\log T,r\in\{s_0,s_0'\})\right]\\
&\ll\sum_{n> N}\frac{1}{n!}|\xi +\xi'|^n\frac{(\log\log T)^n}{\tilde{\mathfrak{s}}^n}.
\end{align*}
To estimate the last line, observe that for \(n\geq 10(|z|+1)\), \[\left|e^z-\sum_{j=0}^n\frac{z^j}{j!}\right|\leq\sum_{j=n+1}^\infty\frac{|z|^j}{j!}\leq\frac{|z|^n}{n!}\leq\frac{|z|^n}{(\frac{n}{e})^n}\] by Stirling's formula, and is furthermore less than or equal to \(e^{-n}\) by our assumption on the size of \(|z|\) with respect to \(n\). We can therefore take \(z=\xi(\log\log T)^{1/2}\) where \(|\xi|\leq C\sqrt{\log\log T}\) (and similarly for \(\xi'\)), \(N\geq C\log\log T\) for some constant \(C\) to get \begin{align}
\sum_{n> N}\frac{1}{n!}|\xi +\xi'|^n\frac{(\log\log T)^n}{\tilde{\mathfrak{s}}^n}&\ll\sum_{n> N}\frac{1}{n!}\left(C\sqrt{\log\log T}\right)^n\left(2\log\log T\right)^{n/2}\label{eq:fouriererror}\\
&\ll e^{-\sqrt{2}C\log\log T}.\nonumber
\end{align}
By redefining the constant \(C\), this is \((\log T)^{-C}\).
We will finish the proof by introducing a lemma that shows that the difference in expectations between the Fourier transforms of the prime sums and that of a Gaussian vector is small.

\begin{lemma}\label{lma:fourierexpectations}
\begin{align*}
\left|\mathbb{E}\left[\exp\left(\frac{i\xi P_1(s_0)+i\xi' P_1(s_0')}{\tilde{\mathfrak{s}}}\right)\right]-\mathbb{E}[\exp(i\xi \mathcal{\tilde{Z}}+i\xi'\mathcal{\tilde{Z}}')]\right|&\ll\frac{1}{(\log\log T)^2}.
\end{align*}
\end{lemma}

\begin{proof}
In order to compare moments, we will truncate at \(N\), which we do by splitting
\begin{align*}
&\left|\mathbb{E}\left[\exp\left(\frac{i\xi P_1(s_0)+i\xi' P_1(s_0')}{\tilde{\mathfrak{s}}}\right)\right]-\mathbb{E}[\exp(i\xi \mathcal{\tilde{Z}}+i\xi' \mathcal{\tilde{Z}}')]\right|\\
&=\left|\mathbb{E}\left[\exp\left(\frac{i\xi P_1(s_0)+i\xi' P_1(s_0')}{\tilde{\mathfrak{s}}}\right)\cdot\mathbbm{1}(|P_1(r)|\leq\log\log T,r\in\{s_0,s_0'\})\right]-\mathbb{E}[\exp(i\xi \mathcal{\tilde{Z}}+i\xi' \mathcal{\tilde{Z}}')]\right.\\
&\phantom{=}+\left.\mathbb{E}\left[\exp\left(\frac{i\xi P_1(s_0)+i\xi' P_1(s_0')}{\tilde{\mathfrak{s}}}\right)\cdot\mathbbm{1}(|P_1(s_0)|>\log\log T\text{ or }|P_1(s_0')|>\log\log T)\right]\right|\\
&\ll\left|\mathbb{E}\left[\exp\left(\frac{i\xi P_1(s_0)+i\xi' P_1(s_0')}{\tilde{\mathfrak{s}}}\right)\cdot\mathbbm{1}(|P_1(r)|\leq\log\log T,r\in\{s_0,s_0'\})\right]-\mathbb{E}[\exp(i\xi \mathcal{\tilde{Z}}+i\xi' \mathcal{\tilde{Z}}')\right|\\
&\phantom{=}+\frac{1}{\log T},
\end{align*}
where the last line follows from \begin{align*}
&\mathbb{E}\left[\exp\left(\frac{i\xi P_1(s_0)+i\xi' P_1(s_0')}{\tilde{\mathfrak{s}}}\right)\cdot\mathbbm{1}(|P_1(s_0)|>\log\log T\text{ or }|P_1(s_0')|>\log\log T)\right]\\
&\ll\mathbb{P}(|P_1(s_0)|>\log\log T\text{ or }|P_1(s_0')|>\log\log T)\\
&\ll\frac{1}{\log T}.
\end{align*}
We then compute a truncation estimate for \((\mathcal{\tilde{Z}},\mathcal{\tilde{Z}}')\) similar to the way we did in \eqref{eq:fouriererror} for \((P_1(s_0),P_1(s_0'))\) by writing \begin{align*}
\sum_{n> N}\frac{1}{n!}i^n\mathbb{E}[(\xi \mathcal{\tilde{Z}}+\xi' \mathcal{\tilde{Z}}')^n]&\ll\sum_{n> N}\frac{1}{n!}|\xi+\xi'|^n\frac{n!}{2^{n/2}(n/2)!}(\tfrac{1}{2}\log\log T)^{n/2}\\
&\ll\sum_{n> N}\frac{1}{(n/2)!}\left(C\sqrt{\log\log T}\right)^n\left(\frac{1}{2}\log\log T\right)^{n/2}.
\end{align*} Again, we estimate this sum by taking \(N\geq C\log\log T\), so that \(N\) is large enough that we can apply Stirling's formula to get \((\log T)^{-C'}\) for some constant \(C'\). Thus, we see that the difference in moments is of order \begin{align*}
&\ll\left|\sum_{n\leq N}\frac{1}{n!}i^n\mathbb{E}\left[\left(\frac{\xi P_1(s_0)+\xi' P_1(s_0')}{\tilde{\mathfrak{s}}}\right)^n\cdot\mathbbm{1}(|P_1(r)|\leq\log\log T,r\in\{s_0,s_0'\})\right]\right.\\
&\phantom{\ll}\left.-\sum_{n\leq N}\frac{1}{n!}i^n\mathbb{E}[(\xi \mathcal{\tilde{Z}}+\xi' \mathcal{\tilde{Z}}')^n]\right|+(\log T)^{-D},
\end{align*}
where \(D=\min\{C,C'\}\).
In order to compare moments without the event \(\{|P_1(r)|\leq\log\log T,r\in\{s_0,s_0'\}\}\), we reinsert the event \(\{|P_1(s_0)|>\log\log T\text{ or }|P_1(s_0')|>\log\log T\}\) by adding and subtracting to obtain
\begin{align}
&\left|\sum_{n\leq N}\frac{1}{n!}i^n\mathbb{E}\left[\left(\frac{\xi P_1(s_0)+\xi' P_1(s_0')}{\tilde{\mathfrak{s}}}\right)^n\right]-\sum_{n\leq N}\frac{1}{n!}i^n\mathbb{E}[(\xi \mathcal{\tilde{Z}}+\xi' \mathcal{\tilde{Z}}')^n]\right.\label{eq:fourierdifference}\\
&\left.\phantom{\ll}-\sum_{n\leq N}\frac{1}{n!}i^n\mathbb{E}\left[\left(\frac{\xi P_1(s)+\xi' P_1(s')}{\tilde{\mathfrak{s}}}\right)^n\cdot\mathbbm{1}(|P_1(s_0)|>\log\log T\text{ or }|P_1(s_0')|>\log\log T)\right]\right|+(\log T)^{-D}\nonumber
\end{align}
where \begin{align*}
&\sum_{n\leq N}\frac{1}{n!}i^n\mathbb{E}\left[\left(\frac{\xi P_1(s_0)+\xi' P_1(s_0)}{\tilde{\mathfrak{s}}}\right)^n\cdot\mathbbm{1}(|P_1(s_0)|>\log\log T\text{ or }|P_1(s_0')|>\log\log T)\right]\\
&\leq\sum_{n\leq N}\frac{1}{n!}i^n\mathbb{E}\left[\left(\frac{\xi P_1(s_0)+\xi'P_1(s_0')}{\tilde{\mathfrak{s}}}\right)^{2n}\right]^{1/2}\cdot\mathbb{E}[\mathbbm{1}(|P_1(s_0)|>\log\log T\text{ or }|P_1(s_0')|>\log\log T)]^{1/2}\\
&\ll\sum_{n\leq N}\frac{1}{n!}\mathbb{E}\left[\left(\frac{\xi P_1(s_0)+\xi' P_1(s_0')}{\tilde{\mathfrak{s}}}\right)^{2n}\right]^{1/2}\cdot\frac{1}{\log T}\ll\frac{1}{\log T}
\end{align*}
by the Cauchy-Schwarz inequality. We finish by applying the moment estimate \cref{lma:Pmoments}, so that for \(n\) even \eqref{eq:fourierdifference} becomes
\begin{align*}
&\left|\sum_{n\leq N}\frac{1}{n!}i^n\left\{\frac{n!}{(n/2)!2^{n/2}}\left[\frac{V(\xi,\xi')^{n/2}+O_n(V(\xi,\xi')^{n/2-2})}{\tilde{\mathfrak{s}}^n}-\frac{V(\xi,\xi')^{n/2}}{\tilde{\mathfrak{s}}^n}\right]\right.\right.\\
&\left.\left.\phantom{\ll}+O\left(\frac{Y^{2n}}{T}\right)\right\}\right|+\frac{1}{\log T}+(\log T)^{-D}.
\end{align*}
Here dividing by the precise variance for \(P_1(s)\) given by \(\tilde{\mathfrak{s}}^2\) allows us to reduce the terms in the brackets to \(O_n(V(\xi,\xi')^{n/2-2})/\tilde{\mathfrak{s}}^n\), and this is similarly all that remains when \(n\) is odd. This leaves us with a difference of moments of order
\begin{align*}
&\frac{V(\xi,\xi')^{N/2-2}}{\tilde{\mathfrak{s}}^N}+\frac{Y^{2N}}{T}+\frac{1}{\log T}+(\log T)^{-D}\\
&\ll\frac{1}{(\log\log T)^2}.
\end{align*}

\end{proof}

\noindent Now, using \cref{lma:fourierestimate} and \cref{lma:fourierexpectations}, we have

\centerline{
  \begin{minipage}{\linewidth}
\begin{align*}
&d_{\mathcal{D}}\left[\left(\frac{|P_1(s_0)|}{\tilde{\mathfrak{s}}},\frac{|P_1(s_0')|}{\tilde{\mathfrak{s}}}\right),\left(\mathcal{\tilde{Z}},\mathcal{\tilde{Z}}'\right)\right]\\
&\leq\frac{1}{F}+\left\{(RF)^2\left|\mathbb{E}\left[\exp\left(\frac{i\xi P_1(s_0)+i\xi' P_1(s_0')}{\tilde{\mathfrak{s}}}\right)\right]-\mathbb{E}[\exp(i\xi\mathcal{\tilde{Z}}+i\xi'\mathcal{\tilde{Z}}')]\right|\right.\\
&\phantom{\leq}\left.+\mathbb{P}\left(\left|\frac{P_1(s_0)}{\tilde{\mathfrak{s}}}\right|>R\text{ or }\left|\frac{P_1(s_0')}{\tilde{\mathfrak{s}}}\right|>R\right)+\mathbb{P}(|\mathcal{\tilde{Z}}|>R\text{ or }|\mathcal{\tilde{Z}}'|>R)\right\}
\end{align*}
\end{minipage}
}
\noindent for any \(R,F>0\). Since \[\mathbb{P}\left(|P_1(s_0)|>\tilde{\mathfrak{s}}(\log\log T)^{1/4}\right)\leq\frac{1}{(\log\log T)^{3/2}}\mathbb{E}[|P_1(s_0)|^2]\ll\frac{1}{\sqrt{\log\log T}}\] we can take \(R=(\log\log T)^{1/4}\) and \(F=C\sqrt{\log\log T}\) to get \begin{align*}
\frac{1}{\sqrt{\log\log T}}+\left[C(\log\log T)^{3/2}\left(\frac{1}{(\log\log T)^2}\right)+\frac{1}{\sqrt{\log\log T}}+\frac{1}{(\log T)^D}\right]\ll\frac{1}{\sqrt{\log\log T}}.
\end{align*}
\end{proof}

\subsection{Proof of \cref{prop:comparingnormals}} \begin{proof}
Similar to the way we proved \crefrange{prop:movingoffaxis}{prop:discardingprimes}, by taking advantage of the Lipschitz nature of the Dudley distance, we obtain

\begin{align}
d_{\mathcal{D}}\left[\left(\mathcal{\tilde{Z}},\mathcal{\tilde{Z}}'\right),\left(\mathcal{Z},\mathcal{Z}'\right)\right]\leq\sup_{f\in\mathcal{L}}\left|\int_{\mathcal{B}}\frac{f(x,x')}{2\pi}\left[\frac{e^{-(x^T\tilde{C}^{-1}x)/2}}{\sqrt{\det\tilde{C}}}-\frac{e^{-(x^TC^{-1}x)/2}}{\sqrt{\det C}}\right]dx\right|,\label{eq:normaldifference}
\end{align}
where \(\tilde{C}\) is the covariance matrix for \((\mathcal{\tilde{Z}},\mathcal{\tilde{Z}}')\).
By \cref{lma:covariance}, \[\frac{1}{\log\log T}\sum_{p\leq Y}\frac{\cos(|h-h'|\log p)}{p}=\alpha+O\left(\frac{1}{\log\log T}\right),\] so we can set \(\tilde{C}=C+\tilde{\mathcal{E}}\) for \(\tilde{\mathcal{E}}=
\begin{pmatrix}
0 & O(1/\log\log T) \\
O(1/\log\log T) & 0
\end{pmatrix}\) and write \begin{align*}
\tilde{C}^{-1}=(C+\tilde{\mathcal{E}})^{-1}=(I+C^{-1}\tilde{\mathcal{E}})^{-1}C^{-1}=(I-C^{-1}\tilde{\mathcal{E}}+O(\tilde{\mathcal{E}}^2))C^{-1}.
\end{align*}
Thus \(\tilde{C}^{-1}=C^{-1}-C^{-1}\tilde{\mathcal{E}}C^{-1}+O(\tilde{\mathcal{E}}^2)\). This allows us to rewrite the exponential term in the integral as \begin{align*}
e^{-(x^T\tilde{C}^{-1}x)/2}&=e^{-(x^TC^{-1}x)/2}\cdot e^{(x^TC^{-1}\tilde{\mathcal{E}}C^{-1}x+O(\tilde{\mathcal{E}}^2))/2}\\
&=e^{-(x^TC^{-1}x)/2}\left(1+\frac{(C^{-1}x)^T\tilde{\mathcal{E}}(C^{-1}x)+O(\tilde{\mathcal{E}}^2)}{2}\right)\\
&=e^{-(x^TC^{-1}x)/2}+e^{-(x^TC^{-1}x)/2}\left(\frac{(C^{-1}x)^T\tilde{\mathcal{E}}(C^{-1}x)+O(\tilde{\mathcal{E}}^2)}{2}\right).
\end{align*}

\noindent Inserting this into \eqref{eq:normaldifference} we see that the first term cancels, and the remaining terms are quadratic in \(x\) and therefore integrable. Hence the distance is \(O(1/\log\log T)\), since both \(\det C\) and \(\det \tilde{C}\) are \(O(1)\).

\end{proof}

%
%

\begin{acks}[Acknowledgments]
I give thanks to Prof. Louis-Pierre Arguin for his unending support and guidance throughout the preparation of this paper. I also thank the reviewer for providing considered and impactful revision suggestions, including pointing out the subtle need for the precise variance introduced in \crefrange{prop:truncatingprimesum}{prop:comparingnormals}. In addition, I would like to thank Emma Bailey for her insightful comments, and everyone else who read through the paper and offered edits. Your feedback is very much appreciated.
\end{acks}
\begin{funding}
Partial support is provided by grants NSF CAREER 1653602 and NSF DMS 2153803.
\end{funding}

\end{document}